\documentclass[10pt,a4paper]{amsart}
\usepackage[utf8]{inputenc}
\author{Hongyi Chu}
\address{Université de Lille, Lille, France}
\email{hchu@uni-osnabrueck.de}

\usepackage[sec,skipfonts,skipmargins]{enrabbrev}
\usepackage{eucal}

\author{Gijs Heuts}
\address{Universiteit Utrecht, Utrecht, Netherlands}
\email{g.s.k.s.heuts@uu.nl}
\urladdr{http://sites.google.com/site/gijsheuts/}

\title{Two Models for the Homotopy Theory of $\infty$-Operads}

\newcommand{\bbO}{\bbOmega}

\newcommand{\DF}{\bbDelta_{\mathbb{F}}}
\newcommand{\DFi}{\DF^{1}}
\newcommand{\DFiop}{\DF^{1,\op}}

\newcommand{\PSeg}{\mathcal{P}_{\txt{Seg}}}
\newcommand{\PSegDF}{\PSeg(\DF)}
\newcommand{\PSegDFi}{\PSeg(\DFi)}

\newcommand{\PSegO}{\PSeg(\bbO)}

\newcommand{\PCS}{\mathcal{P}_{\txt{CS}}}
\newcommand{\PCSDF}{\PCS(\DF)}
\newcommand{\PCSDFi}{\PCS(\DFi)}

\newcommand{\PCSO}{\PCS(\bbO)}

\newcommand{\bbOint}{\bbO_{\txt{int}}}
\newcommand{\DFint}{\bbDelta_{\mathbb{F},\txt{int}}}
\newcommand{\DFiint}{\bbDelta^{1}_{\mathbb{F},\txt{int}}}
\newcommand{\DFelX}{\bbDelta_{\mathbb{F}/X}^{\txt{el}}}
\newcommand{\DFel}{\bbDelta_{\mathbb{F}}^{\txt{el}}}
\newcommand{\bbOel}{\bbO^{\txt{el}}}
\newcommand{\bbOelT}{\bbOel_{/T}}
\newcommand{\DFiintop}{\bbDelta^{1,\op}_{\mathbb{F},\txt{int}}}

\newcommand{\dext}{\partial^{\txt{ext}}}

\newcommand{\xel}{\txt{el}}
\begin{document}

\begin{abstract}
  We compare two models for \iopds{}: the complete Segal operads of
  Barwick and the complete dendroidal Segal spaces of Cisinski and
  Moerdijk. Combining this with comparison results already in the
  literature, this implies that all known models for \iopds{} are
  equivalent --- for instance, it follows that the homotopy theory of
  Lurie's \iopds{} is equivalent to that of dendroidal sets and that of
  simplicial operads.
\end{abstract}

\maketitle

\section{Introduction}
The theory of \emph{operads} is a convenient framework for organizing
a variety of algebraic structures, such as associative and
commutative algebras, or more interestingly algebras which are associative or commutative up to coherent homotopy. For us, operads will by default be
\emph{coloured} operads, i.e. we allow them to have many objects ---
these can be used to describe structures such as enriched categories or a
pair of rings together with a bimodule. Roughly speaking, an operad
$\mathbf{O}$ consists of a set of objects, for each list of objects
$(x_{1},\ldots,x_{n}, y)$ a set of multimorphisms
$\mathbf{O}(x_{1},\ldots,x_{n};y)$ from $(x_{1},\ldots,x_{n})$ to $y$,
equipped with an action of the symmetric group $\Sigma_{n}$ that
permutes the inputs $x_{i}$, and associative and unital composition
operations for the multimorphisms. More generally, we can consider
\emph{enriched} operads, where the sets of multimorphisms are replaced
by objects of some symmetric monoidal category, such as vector spaces
or chain complexes; these can be used to describe algebraic structures
such as Lie algebras or Poisson algebras. 

In topology, we typically encounter operads enriched in topological
spaces (or simplicial sets), such as the $E_{n}$-operads of
May~\cite{MayGeomIter}. There is an evident notion of (weak) homotopy equivalence between such operads and one would like to consider the category of topological operads and weak equivalences as a homotopy theory. Unfortunately,
for many purposes it can be difficult to work with this theory, because topological operads are in a sense too rigid ---
for instance, a weak equivalence between topological operads $\mathbf{P}$ and $\mathbf{Q}$ need not induce an equivalence between the homotopy theories of $\mathbf{P}$-algebras and $\mathbf{Q}$-algebras. Moreover,
one often encounters structures that are naturally seen as operad
algebras in a homotopy-coherent sense, but can be difficult to
rigidify to fit in this framework --- as a baby example, it is
reasonable to think of symmetric monoidal categories as ``commutative
monoids'' in the (2-)category of categories, but actual commutative
monoids require the associativity and symmetry conditions to hold
strictly, which is essentially never true for interesting examples.

For these reasons, it is desirable to have a usable theory of ``weak''
or homotopy-coherent operads, where composition of multimorphisms is
only associative up to a (specified) coherent choice of higher
homotopies, and homotopy-coherent algebras for them. The foundations
for a theory of such \emph{$\infty$-operads} were set up by Lurie in
\cite{HA}; his work gives a powerful framework for working with
homotopy-coherent algebraic structures, as evidenced by the many results obtained in \cite{HA} after building these foundations.

Although Lurie's model is by far the best-developed version of
\iopds{}, a number of other models have been proposed, namely the
\emph{dendroidal sets} of Moerdijk and Weiss~\cite{MoerdijkWeiss} and
the closely related models of complete dendroidal Segal spaces and
dendroidal Segal operads of Cisinski and
Moerdijk~\cite{CisinskiMoerdijkDendSeg}, and the \emph{complete Segal
  operads} of Barwick~\cite{BarwickOpCat}. Moreover, just as
simplicial categories give a model for \icats{}, we can consider
simplicial operads as a model for \iopds{}; appropriate model category
structures on this category have been constructed by Cisinski and
Moerdijk~\cite{CisinskiMoerdijkSimplOpd} and by
Robertson~\cite{RobertsonThesis}.

Some comparisons between these different models are already known:
\begin{itemize}
\item Cisinski and Moerdijk compare the three dendroidal models in
  \cite{CisinskiMoerdijkDendSeg}, and also compare dendroidal sets to
  simplicial operads in \cite{CisinskiMoerdijkSimplOpd}.
\item Barwick compares his model to Lurie's in \cite{BarwickOpCat}.
\item Heuts, Hinich, and Moerdijk obtain a partial comparison between
  dendroidal sets and Lurie's model in
  \cite{HeutsHinichMoerdijkDendrComp}. However, their result is
  restricted to operads without units.
\end{itemize}
In this paper, our goal is to prove one of the missing comparisons: we
will show that the homotopy theory of Barwick's complete Segal operads
is equivalent to that of complete dendroidal Segal spaces. To state a
more precise version of our result, recall that Barwick's Segal
operads are certain presheaves of spaces on a category $\DF$, forming
a full subcategory $\PSegDF$ of the \icat{} $\mathcal{P}(\DF)$ of
presheaves --- we will refer to them as \emph{Segal presheaves} on
$\DF$ to avoid confusion with the Segal operads of Cisinski and
Moerdijk~\cite{CisinskiMoerdijkSimplOpd}, which are a dendroidal
analogue of Segal categories. Similarly, the dendroidal Segal spaces
of Cisinski and Moerdijk are certain presheaves on a category $\bbO$
(we will likewise refer to them as \emph{Segal presheaves} on $\bbO$),
forming a full subcategory $\PSegO$ of the \icat{} $\mathcal{P}(\bbO)$
of all presheaves.  We will define a functor $\tau \colon \DFi \to
\bbO$, where $i \colon \DFi \hookrightarrow \DF$ is a certain full
subcategory, and prove:
\begin{thm}\label{thm:mainthm}
  Composition with the functors $i$ and  $\tau$ induces equivalences of \icats{}
  \[ \PSegO \isoto \PSegDFi \xleftarrow{\sim} \PSegDF.\]
  These functors restrict further to give equivalences between the full
  subcategories of complete objects.
\end{thm}
Here the \emph{complete} objects are those whose underlying Segal
spaces are complete in the sense of Rezk~\cite{RezkCSS}. We will prove that $i$ gives
an equivalence in Lemma~\ref{lem:ieq}, that $\tau$ gives an
equivalence in Theorem~\ref{thm:taueq}, and that we get equivalences
on complete objects in Corollary~\ref{cor:itauCSeq}. In fact, these subcategories of complete objects are localizations of the respective \icats{} of Segal presheaves. A map of dendroidal Segal spaces is known to become an equivalence after completion if and only if it is \emph{fully faithful} and \emph{essentially surjective} by a result of Cisinski and Moerdijk. In Corollary \ref{cor:complete} we apply Theorem \ref{thm:mainthm} to deduce the analogous statement for Barwick's Segal operads.

Combining Theorem~\ref{thm:mainthm} with the above-mentioned
comparison results already in the literature, this implies that all
known models for \iopds{} are equivalent. In particular, we obtain the
following interesting comparisons as an immediate consequence of our
work, answering a question of Lurie~\cite{HA}:
\begin{cor}
  The homotopy theory of Lurie's \iopds{} is equivalent to that of
  dendroidal sets and to that of simplicial
  operads.
\end{cor}

Although we have chosen to use the language of \icats{} in this paper,
as we believe this leads to a cleaner presentation of our work,
our result can also be interpreted in the language of model
categories: the \icats{} $\PSegO$, $\PSegDFi$, and $\PSegDF$ can be
obtained from Bousfield localizations of the projective (or Reedy)
model structures on the categories $\Fun(\bbO^{\op}, \sSet)$,
$\Fun(\DFiop, \sSet)$, and $\Fun(\DF^{\op}, \sSet)$ of simplicial
presheaves on $\bbO$, $\DFi$, and $\DF$, respectively. Moreover, it is
easy to see that composition with $i$ with $\tau$ give right Quillen
functors between these localized model structures (with left adjoints
given by left Kan extensions). In this language, our result says:
\begin{cor}
  The Quillen adjunctions
  \[ \tau_{!} : \txt{Fun}(\DFiop, \sSet) \rightleftarrows
  \Fun(\bbO^{\op}, \sSet) : \tau^{*},\]
  \[ i_{!} : \Fun(\DFiop, \sSet) \rightleftarrows
  \Fun(\DF^{\op}, \sSet) : i^{*},\]
  are Quillen equivalences, where the categories involved are equipped with the Bousfield localizations of the respective projective model structures at the Segal equivalences. Moreover, they remain Quillen equivalences if we
  localize further to get the model structures for complete objects.
\end{cor}
Since a Quillen adjunction is a Quillen equivalence \IFF{} it induces
an equivalence of homotopy categories, this is an immediate
consequence of Theorem~\ref{thm:mainthm}.

\subsection{Overview}
In \S\ref{sec:SegPShDF} we review the definition of Barwick's Segal
operads, which we will call Segal presheaves on $\DF$. We also show
that we can equivalently consider Segal presheaves on a full
subcategory $\DFi$ of $\DF$. Next, in \S\ref{sec:SegPshDendr} we
review the dendroidal Segal spaces of Cisinski and Moerdijk, which we
will similarly refer to as \emph{Segal presheaves} on $\bbO$. In
\S\ref{sec:DFitobbO} we define the functor $\tau$ from $\DFi$ to $\bbO$,
and then in \S\ref{sec:comparison} we prove our main comparison
result, namely that composing with $\tau$ gives an equivalence between
the two \icats{} of Segal presheaves. Finally, in \S\ref{sec:complete}
we review the definition of \emph{complete} Segal presheaves on $\DF$
and $\bbO$, and observe that these agree under our equivalence.

\subsection{Notation}
This paper is written in the language of \icats{} (or more specifically \emph{quasicategories}), as developed by Joyal~\cite{JoyalQCNotes},
Lurie~\cite{HTT,HA} and others. We will use
terminology from \cite{HTT}; here we give a few reminders:
\begin{itemize}
\item $\mathcal{S}$ is the \icat{} of spaces (or
  $\infty$-groupoids).
\item If $\mathcal{C}$ is an \icat{}, we write
  $\mathcal{P}(\mathcal{C})$ for the \icat{}
  $\Fun(\mathcal{C}^{\op},\mathcal{S})$ of presheaves of spaces on
  $\mathcal{C}$.
\item $\simp$ is the usual simplicial indexing category. We say a
  morphism $\phi \colon [n] \to [m]$ is \emph{inert} if it is the
  inclusion of a subinterval in $[m]$, i.e. if $\phi(i) = \phi(0)+i$ for
  all $i$, and \emph{active} if it preserves the end-points, i.e. if
  $\phi(0) = 0$ and $\phi(n) = m$. The active and inert morphisms form
  a factorization system on $\simp$.
\end{itemize}

\subsection{Acknowledgments}
This project began through discussions between the second author and
Fernando Muro during a visit to the University of Sevilla. The first author thanks Denis-Charles Cisinski, David Gepner, and especially his PhD advisor Markus Spitzweck for helpful discussions. Moreover, he wishes to acknowledge the support of the DFG during his PhD and of the Labex CEMPI (ANR-11-LABX-0007-01) during his postdoc. The third author was supported by the European Research Council (ERC) under the European Union's Horizon 2020 research and innovation programme (grant agreement No. 682922).

\section{Segal Presheaves on \texorpdfstring{$\DF$}{DF} and \texorpdfstring{$\DFi$}{DFi}}\label{sec:SegPShDF}
In this section we review the model for \iopds{} introduced by Barwick
in \cite{BarwickOpCat}, which we will refer to as \emph{Segal
  presheaves on $\DF$}. We also show that these are equivalent to
Segal presheaves on a full subcategory $\DFi$, which will be easier to
relate to the dendroidal category later on.

\begin{defn}
  Write $\mathbb{F}$ for a skeleton of the category of finite sets (possibly
  empty), i.e. the category with objects $\mathbf{k} :=
  \{1,\ldots,k\}$, $k = 0,1,\ldots$, and morphisms maps of sets. Let $\DF$ be the category with
  objects pairs $([n], f \colon [n] \to \mathbb{F})$ with a morphism
  $([n], f) \to ([m], g)$ given by a morphism $\phi \colon [n] \to
  [m]$ in $\simp$ and a natural transformation $\eta \colon f \to g
  \circ \phi$ such that
  \begin{enumerate}[(i)]
  \item the map $\eta_{i} \colon f(i) \to g(\phi(i))$ is injective for
    all $i = 0,\ldots,m$,
  \item the commutative square
    \csquare{f(i)}{g(\phi(i))}{f(j)}{g(\phi(j))}{\eta_{i}}{}{}{\eta_{j}}
    is a pullback square for all $0 \leq i \leq j \leq m$.
  \end{enumerate}
  We say an object $([n], f) \in \DF$ has \emph{length} $n$.
\end{defn}

\begin{notation}
  If $([n], f)$ is an object of $\DF$, we will write $f^{ij} \colon
  f(i) \to f(j)$ for the image under $f$ of the map $i \to j$ in
  $\mathbf{n}$; we abbreviate $f^{i(i+1)}$ to $f^{i+1}$.
\end{notation}

\begin{remark}
  An object of $\DF$ is thus a sequence 
  \[ \mathbf{k}_{0} \to \mathbf{k}_{1} \to \cdots \to
  \mathbf{k}_{n} \]
  of maps of finite sets. If $\mathbf{k}_{n} = \mathbf{1}$, we can
  think of this as a tree with levels: we think of the elements of the
  sets $\mathbf{k}_{i}$ as the edges of the tree --- in particular,
  $\mathbf{k}_{0}$ is the set of leaves, and the map $\mathbf{k}_{i}
  \to \mathbf{k}_{i+1}$ assigns to an edge $e$ in level $i$ the unique
  outgoing edge of the vertex that has $e$ as an incoming edge; thus
  we can also think of the elements of $\mathbf{k}_{i}$ with $i > 0$
  as the vertices of the tree. A general object of $\DF$ can then be
  thought of as a ``forest'', i.e. a collection of trees indexed by
  $\mathbf{k}_{n}$. To define Segal presheaves we now want to impose
  relations on presheaves on $\DF$ that force the value on a forest  to
  decompose into the values at the basic corollas, corresponding to
  the objects $([1], \mathbf{n} \to \mathbf{1})$, as well as the single
  edge, $([0], \mathbf{1})$.
\end{remark}

\begin{remark}
  Since morphisms in $\DF$ are required to induce pullback
  squares, given an object $([m], f) \in \DF$ and a morphism $q
  \colon \mathbf{a} \to f(m)$, there exists an essentially unique
  morphism $([m], f_{\mathbf{a}}) \to ([m], f)$ over $\id_{[m]}$
  with value $q$ at $m$.  
\end{remark}

\begin{remark}
  The projection $\DF \to \simp$ is a Grothendieck fibration: given
  $([n], f)$ and $\phi \colon [m] \to [n]$, the map $\phi^{*}([n], f)
  := ([m], f \circ \phi) \to ([n], f)$ is a Cartesian morphism. In
  general a morphism $(\phi, \eta) \colon ([m], g) \to ([n], f)$ is
  Cartesian \IFF{} $\eta_{i} \colon g(i) \to f(\phi(i))$ is an
  isomorphism for all $i$. 
\end{remark}

\begin{defn}
  We say a map $(\phi, \eta) \colon ([n], f) \to ([m], g)$ in $\DF$ is
  \begin{enumerate}
  \item \emph{injective} if $\phi \colon [n] \to [m]$ is injective,
  \item \emph{surjective} if $\phi$ is surjective,
    and $\eta_{i} \colon f(i) \to g(\phi(i))$ is an isomorphism for
    all $i$ (or equivalently, if $\phi$ is surjective and $(\phi, \eta)$
    is Cartesian),
  \item \emph{inert} if $\phi$ is inert in $\simp$,
  \item \emph{active} if $\phi$ is active in $\simp$, and $\eta_{i}
    \colon f(i) \to g(\phi(i))$ is an isomorphism for all $i$ (or
    equivalently, if $\phi$ is active and $(\phi, \eta)$ is
    Cartesian).
  \end{enumerate}
  The surjective and injective maps, as well as the active and inert
  maps, form factorization systems on $\DF$ --- this is clear since
  they are both lifted from factorization systems on $\simp$ via the
  fibration $\DF \to \simp$. We write $\DFint$ for the subcategory of
  $\DF$ containing only the inert maps. 
\end{defn}

A presheaf $\mathcal{F} \colon \DF^{\op} \to \mathcal{S}$ is a
\emph{Segal presheaf} if it satisfies the following three ``Segal conditions'':
\begin{enumerate}[(1)]
\item for every object $([n], \mathbf{k}_{0} \to \mathbf{k}_{1} \to \cdots \to
\mathbf{k}_{n})$ of $\DF$, the natural map
\[ \mathcal{F}([n], \mathbf{k}_{0} \to \cdots \to \mathbf{k}_n) \to
\mathcal{F}([1], \mathbf{k}_0 \to \mathbf{k}_1)
\times_{\mathcal{F}([0], \mathbf{k}_1)} \cdots
\times_{\mathcal{F}([0], \mathbf{k}_{n-1})} \mathcal{F}([1], \mathbf{k}_{n-1} \to \mathbf{k}_n)
\]
is an equivalence,
\item for every object $([1], \mathbf{k} \to \mathbf{l})$, the natural map
\[ \mathcal{F}([1], \mathbf{k} \to \mathbf{l}) \to \prod_{i = 1}^{l}
\mathcal{F}([1], \mathbf{k}_{i} \to \mathbf{1})\] (where $\mathbf{k}_{i}$
is the fibre of $\mathbf{k}$ over $i \in \mathbf{l}$) is an
equivalence,
\item for every object $([0], \mathbf{k})$, the natural map
  \[ \mathcal{F}([0], \mathbf{k}) \to \prod_{i = 1}^{k}
  \mathcal{F}([0], \mathbf{1})\]
  is an equivalence.
\end{enumerate}
For us, a more convenient formulation of this definition will be the following:
\begin{defn}
  Let $\DF^{\txt{el}}$ denote the full subcategory of $\DFint$ spanned
  by the objects $([1], \mathbf{k} \to \mathbf{1})$, for $k \geq 0$, and $([0],
  \mathbf{1})$, which we refer to as \emph{elementary objects}. Then we say a presheaf $\mathcal{F} \colon \DF^{\op}
  \to \mathcal{S}$ is a \emph{Segal presheaf} if the restriction
  $\mathcal{F}|_{\DFint^{\op}}$ is the right Kan extension of its
  restriction to $\DF^{\txt{el},\op}$. We write $\PSegDF$ for the full
  subcategory of $\mathcal{P}(\DF)$ spanned by the Segal presheaves.
\end{defn}

\begin{remark}
  For $X \in \DF$, write $\DFelX$ for the category $\DF^{\txt{el}}
  \times_{\DFint} (\DFint)_{/X}$; then $\mathcal{F} \in
  \mathcal{P}(\DF)$ is a Segal presheaf \IFF{} for every $X \in \DF$
  the map $\mathcal{F}(X) \to \lim_{E \in \DFelX} \mathcal{F}(E)$ is
  an equivalence. If we let $X_{\Seg}$ denote the presheaf $\colim_{E
   \in \DFelX} E$ (where we regard $E$ as a presheaf via the
  Yoneda embedding), then this means that $\mathcal{F}$ is a Segal
  presheaf \IFF{} it is local with respect to the maps $X_{\Seg} \to
  X$ for $X \in \DF$. Thus $\PSegDF$ is the localization of
  $\mathcal{P}(\DF)$ with respect to these maps --- in particular, it
  is an accessible localization of $\mathcal{P}(\DF)$; we write
  $L_{\DF} \colon \mathcal{P}(\DF) \to \PSegDF$ for the localization
  functor. We call the local equivalences for this localization,
  i.e. the maps that are sent to equivalences by $L_{\DF}$, the
  \emph{Segal equivalences} in $\mathcal{P}(\DF)$.
\end{remark}

\begin{defn}
  Let $\DFi$ be the full subcategory of $\DF$ spanned by the objects
  $([n], f)$ such that $f(n) = \mathbf{1}$. The active-inert and
  surjective-injective factorization systems on $\DF$ clearly
  restrict to factorization systems on $\DFi$. We write $\DFiint$
  for the subcategory of $\DFi$ containing only the inert maps. Since
  $\DF^{\txt{el}}$ is a full subcategory of $\DFiint$, we can again
  define a presheaf $\mathcal{F} \colon \DFiop \to \mathcal{S}$ to be
  a \emph{Segal presheaf} if the restriction
  $\mathcal{F}|_{\DFiintop}$ is a right Kan extension of its
  restriction to $\DF^{\txt{el},\op}$.  Let $i \colon \DFi
  \hookrightarrow \DF$ denote the inclusion. Then it is clear from the
  definition that composition with $i$ induces a functor $i^{*} \colon
  \PSegDF \to \PSegDFi$.
\end{defn}

\begin{remark}
  A presheaf $\mathcal{F} \in
  \mathcal{P}(\DFi)$ is again a Segal presheaf \IFF{} $\mathcal{F}$ is local
  with respect to the maps $X_{\Seg} \to X$ for $X \in \DFi$. Thus
  $\PSegDFi$ is the localization of $\mathcal{P}(\DFi)$ with respect to
  these maps --- in particular, it is an accessible localization of
  $\mathcal{P}(\DFi)$; we write $L_{\DFi} \colon \mathcal{P}(\DFi) \to
  \PSegDFi$ for the localization functor. We call the local equivalences
  for this localization, i.e. the maps that are sent to equivalences
  by $L_{\DFi}$, the \emph{Segal equivalences} in
  $\mathcal{P}(\DFi)$.
\end{remark}

\begin{lemma}\label{lem:ieq}
  The functor $i^{*} \colon \PSegDF \to \PSegDFi$ is an equivalence.
\end{lemma}
\begin{proof}
  We will show that the right Kan extension functor $i_{*} \colon
  \mathcal{P}(\DFi) \to \mathcal{P}(\DF)$, which is right adjoint to
  $i^{*} \colon \mathcal{P}(\DF) \to \mathcal{P}(\DFi)$, restricts to
  an inverse to $i^{*}$ on Segal presheaves. First of all, as $i^{*}$
  preserves colimits, it is easy to see that it sends Segal
  equivalences to Segal equivalences; it follows that $i_*$ preserves
  the property of being a Segal presheaf. To see that $i_*$ indeed
  gives the desired inverse, we will show that the natural
  transformations $\id_{\mathcal{P}(\DF)} \to i_{*}i^{*}$ and
  $i^{*}i_{*} \to \id_{\mathcal{P}(\DFi)}$ are equivalences on Segal
  presheaves.

  Since $i \colon \DFi \hookrightarrow \DF$ is the inclusion of a full
  subcategory, the functor $i_{*}$ is fully faithful, and so
  $i^{*}i_{*} \to \id_{\mathcal{P}(\DFi)}$ is an equivalence for any
  presheaf on $\DFi$.

  If $\mathcal{F}$ is a presheaf on $\DF$, then the component of $\mathcal{F} \to i_{*}i^{*}\mathcal{F}$ at $([n],
  f) \in \DF$ is the natural map $\mathcal{F}([n], f) \to
  \lim_{((\DFi)_{/([n], f)})^{\op}} \mathcal{F}$. If $f(n) = 
  \mathbf{k}$, let $([n], f_{i})$ for $i = 1,\ldots,k$ denote the
  subtree
  \[ f(0)_{i} \to f(1)_{i} \to \cdots \to f(n-1)_{i} \to \{i\}\] given
  by the fibres $f(j)_{i}$ of $f^{jn} \colon f(j) \to f(n) =
  \mathbf{k}$ at $i \in \mathbf{k}$. To understand the limit
  $i_{*}i^{*}\mathcal{F}([n], f)$ we will prove that the inclusion
  $\{f_{i} : i = 1,\ldots,k\} \hookrightarrow (\DFi)_{/([n], f)}$ is
  cofinal. By Joyal's version of Quillen's Theorem A \cite[Theorem 4.1.3.1]{HTT} it suffices to show that for
  every $(\phi, \eta) \colon ([m], g) \to ([n], f)$ in $(\DFi)_{/([n],
    f)}$, the category $\{f_{i} : i = 1,\ldots,k\}_{(\phi,\eta)/}$ is
  weakly contractible. But it is clear that this is the one-object set
  $\{f_{j}\}$, where $j = f^{\phi(m)n}(\eta_{m}(1))$, which is
  certainly weakly contractible.

  It follows that $(i_{*}i^{*}\mathcal{F})([n], f) \simeq \prod_{i \in
    f(n)} \mathcal{F}([n], f_{i})$, and the map $\mathcal{F}([n], f)
  \to (i_{*}i^{*}\mathcal{F})([n], f)$ is the natural map
  $\mathcal{F}([n], f) \to \prod_{i \in f(n)} \mathcal{F}([n],
  f_{i})$. But if $\mathcal{F}$ is a Segal presheaf then this map is
  an equivalence (since $(\DF^{\txt{el}})_{/([n], f)}$ is the
  coproduct of $(\DF^{\txt{el}})_{/([n], f_{i})}$ over $i \in f(n)$).
\end{proof}

\section{Segal Presheaves on \texorpdfstring{$\bbO$}{O}}\label{sec:SegPshDendr}
We now recall the definition of the dendroidal category $\bbO$. It
was originally defined by Moerdijk and Weiss~\cite{MoerdijkWeiss} as a
category of trees, with the morphisms given by the operad maps between
the free operads generated by these trees. A combinatorial reformulation of this definition was later given by
Kock~\cite{KockTree}, and it is his definition that we will recall
here.

\begin{defn}
  A \emph{polynomial endofunctor} is a diagram of sets
  \[ X_{0} \xfrom{s} X_{2} \xto{p} X_{1} \xto{t} X_{0}.\]
  A polynomial endofunctor is a \emph{tree} if:
  \begin{enumerate}[(1)]
  \item The sets $X_{i}$ are all finite.
  \item The function $t$ is injective.
  \item The function $s$ is injective, with a unique element $R$ (the
    \emph{root}) in the complement of its image.
  \item Define a successor function $\sigma \colon X_{0} \to X_{0}$ as follows. First, set
    $\sigma(R) = R$. For $e \in s(X_{2})$ (which is the complement of $R$ in $X_{0}$), take $e'$ in $X_{2}$ with $s(e') = e$ and set $\sigma(e) = t(p(e'))$. Then for every $e$ there exists some $k$ such that
    $\sigma^{k}(e) = R$.
  \end{enumerate}
\end{defn}

\begin{remark}
  The intuition behind this notion of ``tree'' is as follows: we think
  of $X_{0}$ as the set of edges of the tree, $X_{1}$ as the set of
  vertices (our trees do not have vertices at their leaves or root),
  and $X_{2}$ as the set of pairs $(v, e)$ where $v$ is a vertex and
  $e$ is an incoming edge of $v$. The function $s$ is the projection
  $s(v,e) = e$, the function $p$ is the projection $p(v,e) = v$, and
  the function $t$ assigns to each vertex its unique outgoing edge.
\end{remark}

\begin{remark}
  The name ``polynomial endofunctor'' comes from the fact that such a
  diagram induces an endofunctor of $\Set_{/X_{0}}$ given by
  $t_{!}p_{*}s^{*}$. We refer the reader to \cite{KockTree} for more
  discussion of this.
\end{remark}

\begin{defn}
  A morphism of polynomial endofunctors $f \colon X \to Y$ is a
  commutative diagram
\[
\begin{tikzcd}
  X_{0} \arrow{d}{f_{0}}& X_{2} \arrow{l} \arrow{r} \arrow{d}{f_{2}}& X_{1} \arrow{r}\arrow{d}{f_{1}} & X_{0} \arrow{d}{f_{0}}\\
Y_{0} & Y_{2} \arrow{l} \arrow{r} & Y_{1} \arrow{r} & Y_{0}
\end{tikzcd}
\]
such that the middle square is Cartesian. We write $\bbOint$
for the category of trees and  morphisms of polynomial endofunctors
between them; we will refer to these as the \emph{inert} morphisms
between trees, or as \emph{embeddings} of subtrees.
\end{defn}
\begin{remark}
  By \cite[Proposition 1.1.3]{KockTree} every morphism of polynomial
  endofunctors between trees is injective, which justifies calling
  these morphisms embeddings.
\end{remark}

The following two definitions fix some terminology which we will need later.

\begin{defn}
Let $X$ be a tree. Then a \emph{leaf} of $X$ is an element of $X_0$
which is not in the image of $t: X_1 \rightarrow X_0$.
\end{defn}

\begin{defn}
  We write $C_{n}$ for the \emph{$n$-corolla}, namely the tree
  \[ \{0,1,\ldots,n\} \hookleftarrow \{1,\ldots,n\} \to \{0\}
  \hookrightarrow \{0,1,\ldots,n\}.\] We write $\eta$ for the
  \emph{edge}, namely the trivial tree
  \[ * \hookleftarrow \emptyset \to \emptyset
  \hookrightarrow *.\]
\end{defn}

\begin{defn}
  If $T$ is a tree, let $\txt{sub}(T)$ be the set of subtrees of $T$,
  i.e. the set of morphisms $T' \to T$ in $\bbOint$, and let
  $\txt{sub}'(T)$ be the set of subtrees of $T$ with a marked leaf,
  i.e. the set of pairs of morphisms $(\eta \to T', T' \to T)$, where the image of the first map is a leaf of $T'$. We then
  write $\overline{T}$ for the polynomial endofunctor
  \[ T_{0} \from \txt{sub}'(T) \to \txt{sub}(T) \to T_{0},\]
  where the first map sends a marked subtree to its marked edge, the
  second is the obvious projection, and the third sends a subtree to
  its root.
\end{defn}

\begin{defn}
  The category $\bbO$ has objects trees, and has as morphisms $T \to
  T'$ the morphisms of polynomial endofunctors $\overline{T} \to
  \overline{T}'$.
\end{defn}

\begin{remark}
  By \cite[Corollary 1.2.10]{KockTree}, the polynomial endofunctor
  $\overline{T}$ is in fact the free polynomial monad generated by
  $T$, and the category $\bbO$ is a full subcategory of the Kleisli
  category of the monad assigning the free polynomial monad to a polynomial endofunctor. This means that a
  morphism $\overline{T} \to \overline{T}'$ is uniquely determined by
  the composite $T \to \overline{T} \to \overline{T}'$. In fact, more
  is true:
\end{remark}
\begin{lemma}[\cite{KockTree}*{Lemma 1.3.5}]\label{lem:moredge}
  Any morphism $\overline{T} \to \overline{T}'$ in $\bbO$ is uniquely
  determined by the underlying map $T_{0} \to T'_{0}$ on sets of edges.
\end{lemma}

\begin{defn}
  It follows that $\bbOint$ is a subcategory of $\bbO$; we say a
  morphism in $\bbO$ is \emph{inert} if it lies in the image of
  $\bbOint$.  We also say a morphism $\phi \colon T \to T'$ in $\bbO$
  is \emph{active} if it takes the maximal subtree to the maximal
  subtree, or equivalently if it takes the leaves of $T$ to the leaves
  of $T'$ (bijectively) and the root of $T$ to the root of $T'$.
\end{defn}

\begin{remark}
  In \cite{KockTree} the inert morphisms are called \emph{free}, and
  the active ones \emph{boundary-preserving}. Our terminology follows 
  that of Barwick~\cite{BarwickOpCat} and Lurie~\cite{HA}. 
\end{remark}

\begin{propn}[Kock, {\cite[Proposition 1.3.13]{KockTree}}]
  The active and inert morphisms form a factorization system on $\bbO$.\qed
\end{propn}

\begin{defn}
  Let $\bbO^{\txt{el}}$ be the full subcategory of $\bbOint$ spanned
  by the objects $C_{n}$ ($n = 0,1,\ldots$) and
  $\eta$. We say a presheaf $\mathcal{F} \colon \bbO^{\op} \to
  \mathcal{S}$ is a \emph{Segal presheaf} if the restriction
  $\mathcal{F}|_{\bbOint^{\op}}$ is a right Kan extension of its
  restriction to $\bbO^{\txt{el},\op}$. We write $\PSegO$ for the full
  subcategory of $\mathcal{P}(\bbO)$ spanned by the Segal presheaves.
\end{defn}

\begin{remark}
  For $T \in \bbO$, write $\bbOelT$ for the category $\bbO^{\txt{el}}
  \times_{\bbOint} (\bbOint)_{/X}$; then a presheaf $\mathcal{F} \in
  \mathcal{P}(\bbO)$ is a Segal presheaf \IFF{} the map
  $\mathcal{F}(T) \to \lim_{E \in \bbOelT} \mathcal{F}(E)$ is an
  equivalence for every $T \in \bbO$. If we let $T_{\Seg}$ denote the
  presheaf $\colim_{E \in \bbOelT} E$ (where we regard $E$ as a
  presheaf via the Yoneda embedding), then this means that
  $\mathcal{F}$ is a Segal presheaf \IFF{} it is local with respect to
  the maps $T_{\Seg} \to T$ for $T \in \bbO$. Thus $\PSegO$ is the
  localization of $\mathcal{P}(\bbO)$ with respect to these maps ---
  in particular, it is an accessible localization of
  $\mathcal{P}(\bbO)$; we write $L_{\bbO} \colon \mathcal{P}(\bbO) \to
  \PSegO$ for the localization functor. We call the local equivalences
  for this localization the \emph{Segal equivalences} in
  $\mathcal{P}(\bbO)$.
\end{remark}

\begin{remark}
  The $\infty$-category $\PSegO$ corresponds to the model category of
  \emph{dendroidal Segal spaces} studied by Cisinski and Moerdijk
  \cite{CisinskiMoerdijkDendSeg}.
\end{remark}

\section{From \texorpdfstring{$\DFi$}{DFi} to \texorpdfstring{$\bbO$}{O}}\label{sec:DFitobbO}
In this section we will define a functor $\tau \colon \DFi \to \bbO$.
On objects, the functor $\tau$ takes an object $([n], f)$ in $\DFi$ to
the diagram
\[ \coprod_{i = 0}^{n} f(i) \xfrom{s} \coprod_{i = 0}^{n-1} f(i)
\xto{p} \coprod_{i = 1}^{n} f(i) \xto{t} \coprod_{i = 0}^{n} f(i),\]
where $s$ and $t$ are the obvious inclusions and $p$ takes $x \in
f(i)$ to $f^{i+1}(x) \in f(i+1)$.

If $(\phi, \eta) \colon ([n], f) \to ([m], g)$ is an inert map in
$\DFi$, then we define $\tau(\phi, \eta)$ to be the obvious morphism 
\[
\begin{tikzcd}
  \coprod_{i = 0}^{n} f(i) \arrow{d}& \coprod_{i = 0}^{n-1} f(i)
  \arrow{l} \arrow{r} \arrow{d}& \coprod_{i = 1}^{n} f(i)
  \arrow{r}\arrow{d} & \coprod_{i = 0}^{n} f(i) \arrow{d}\\
\coprod_{j = 0}^{m} g(j) & \coprod_{j = 0}^{m-1} g(j) \arrow{l} \arrow{r} & \coprod_{j = 1}^{m} g(j) \arrow{r} &\coprod_{j = 0}^{m} g(j).
\end{tikzcd}
\]
Here the middle square is Cartesian, as required, since by definition $\eta$ is a
Cartesian natural transformation.

To define $\tau$ for a general map in $\DFi$, it is convenient to first
introduce an intermediate object $\tilde \tau([n],f)$ between $\tau([n], f)$ and its free
monad $\overline{\tau}([n], f)$:
\begin{defn}
  For $([n], f) \in \DFi$, let $\txt{sub}_{\DFi}([n], f)$ denote the set
  of subtrees of $([n], f)$ given by maps in $\DFi$, i.e. the set of
  inert maps $([m], g) \hookrightarrow ([n], f)$ in $\DFi$, or equivalently the set of pairs $(x \in f(i), 0 \leq j
  \leq i)$, corresponding to the subtree
  \[ f(j)_{x} \to f(j+1)_{x} \to \cdots \to f(i-1)_{x} \to \{x\},\]
  where $f(k)_{x}$ is the fibre of $f^{ki} \colon f(k) \to f(i)$ at
  $x$. Similarly, let $\txt{sub}'_{\DFi}([n], f)$ be the set of
  subtrees in $\txt{sub}_{\DFi}([n], f)$ with a marked leaf, or
  equivalently the set of triples $(x \in f(i), 0 \leq j \leq i, y \in
  f(j)_{x})$. We then let $\widetilde{\tau}([n], f)$ denote the polynomial endofunctor
  \[ \coprod_{i = 0}^{n} f(i) \from \txt{sub}'_{\DFi}([n], f) \to
  \txt{sub}_{\DFi}([n], f) \to \coprod_{i = 0}^{n} f(i),\] where the
  first map takes $(x \in f(i), j, y \in f(j)_{x})$ to the marked leaf
  $y$ and the second projects it to $(x\in f(i),j)$, and the third takes the
  subtree $(x \in f(i), j)$ to its root $x$. The definition of $\tau$
  on inert maps clearly gives an injective map $\widetilde{\tau}([n],
  f) \hookrightarrow \overline{\tau}([n], f)$ of polynomial
  endofunctors, and the canonical map $\tau([n], f) \to
  \overline{\tau}([n], f)$ factors through this.
\end{defn}

For a general map $(\phi, \eta) \colon ([n], f) \to ([m], g)$, we
then define a map of polynomial endofunctors $\tau([n], f) \to
\widetilde{\tau}([m], g)$, i.e.
\[
\begin{tikzcd}
\coprod_{i = 0}^{n} f(i) \arrow{d}& \coprod_{i = 0}^{n-1} f(i)
 \arrow{l} \arrow{r} \arrow{d}& \coprod_{i = 1}^{n} f(i)
 \arrow{r}\arrow{d} & \coprod_{i = 0}^{n} f(i) \arrow{d}\\
\coprod_{j = 0}^{m} g(j) & \txt{sub}'_{\DFi}([m], g)  \arrow{l}
\arrow{r} & \txt{sub}_{\DFi}([m], g) \arrow{r} &\coprod_{j = 0}^{m} g(j),
\end{tikzcd}
\]
as follows:
\begin{itemize}
\item The component $\coprod_{i = 0}^{n} f(i) \to \coprod_{j = 0}^{m} g(j)$
is the obvious map, given on $f(i)$ by $\eta_{i} \colon f(i) \to
g(\phi(i))$.
\item  The component $\coprod_{i = 1}^{n} f(i) \to
\txt{sub}_{\DFi}([m], g)$ is given by 
\[ (x \in f(i)) \mapsto (\eta_{i}(x) \in
g(\phi(i)), \phi(i-1)).\]
\item The component $\coprod_{i = 0}^{n-1} f(i) \to
\txt{sub}'_{\DFi}([m], g)$ is defined by
\[ (x \in f(i)) \mapsto
(\eta_{i+1}(f^{i+1}(x)) \in g(\phi(i+1)), \phi(i), \eta_{i}(x) \in g(\phi(i))_{\eta_{i+1}(f^{i+1}(x))})\]
\end{itemize}
We see that the middle square in the diagram above is then Cartesian,
since $\eta$ is a Cartesian natural transformation, so this does
indeed define a map of polynomial endofunctors. We then define
$\tau(\phi,\eta)$ to be the map $\overline{\tau}([n], f) \to
\overline{\tau}([m], g)$ induced by the composite $\tau([n], f) \to
\widetilde{\tau}([m], g) \hookrightarrow \overline{\tau}([m], g)$. 

\begin{lemma}
  $\tau$ is a functor $\DFi \to \bbO$.
\end{lemma}
\begin{proof}
  Since $\tau$ clearly preserves identities, it remains to check that
  it respects composition, i.e. that for \[([n], f) \xto{(\phi, \eta)}
  ([m], g) \xto{(\psi, \lambda)} ([k], h)\] in $\DFi$ the maps
  $\tau((\psi,\lambda) \circ (\phi, \eta))$ and $\tau(\psi, \lambda)
  \circ \tau(\phi, \eta)$ agree. But by Lemma~\ref{lem:moredge} it
  suffices to show that they are given by the same map on the set of
  edges. By definition, for $\tau(\phi, \eta)$ this is the map
  $\coprod_{i = 0}^{n} f(i) \to \coprod_{j = 0}^{m} g(j)$ given on
  $f(i)$ by $\eta_{i} \colon f(i) \to g(\phi(i))$, so it is evident
  that the two maps agree on the edge sets.
\end{proof}

The definition of $\tau$ immediately implies the following lemma:
\begin{lemma}
  The functor $\tau$ preserves the surjective-injective and
  active-inert factorization systems. \qed
\end{lemma}

\begin{lemma}\label{lem:taueleq}
  The functor $\tau$ restricts to an equivalence $\DFel \to
  \bbOel$. Moreover, for any $X \in \DFi$, it induces an 
  equivalence of categories $\DFelX \to \bbOel_{/\tau(X)}$. \qed
\end{lemma}
\begin{proof}
  The first claim follows immediately from the definition of
  $\tau$. Since every object $E\to \tau(X)$ in $\bbO_{\xel/\tau(X)}$
  lies in the image of $\tau$, the induced functor
  $\simp_{\mathbb{F},\xel/X} \to \bbO_{\xel/\tau(X)}$ is essentially
  surjective. We now observe that a morphism in $\bbO_{\xel/\tau(X)}$
  is a commutative diagram
				\begin{equation*}
				\begin{tikzcd}
				E \ar{rr}\ar{rd} & &E'\ar{ld}\\
				&\tau(X),&
				\end{tikzcd}
				\end{equation*}
				where the horizontal map is either the identity map or the image of a unique map of the form $([0],\mathbf{1})\to ([1],f)$ under $\tau$. This shows that the functor $\simp_{\mathbb{F},\xel/X} \to \bbO_{\xel/\tau(X)}$ is also fully faithful. 
			\end{proof}

\begin{lemma}\ 
  \begin{enumerate}[(i)]
  \item The functor $\tau_{!} \colon \mathcal{P}(\DFi) \to
    \mathcal{P}(\bbO)$ preserves Segal equivalences.
  \item  Composition with $\tau$ restricts to a functor
    $\tau^{*} \colon \PSegO \to \PSegDFi$.
  \item This functor has a left adjoint $L_{\bbO} \circ \tau_{!}
    \colon \PSegDFi \to \PSegO$.
  \end{enumerate}
\end{lemma}
\begin{proof}
  To prove (i) it suffices to show that the images under $\tau_{!}$ of
  the generating Segal equivalences $X_{\Seg} \to X$ in
  $\mathcal{P}(\DFi)$ are Segal equivalences in
  $\mathcal{P}(\bbO)$. But since $\tau_{!}$ preserves colimits,
  Lemma~\ref{lem:taueleq} implies that $\tau_{!}(X_{\Seg}) \simeq (\tau
  X)_{\Seg}$, and so these maps are among the generating Segal equivalences
  for $\mathcal{P}(\bbO)$. Then the claims (ii) and (iii) are immediate
  consequences of (i).
\end{proof}

\begin{lemma}\label{lem:unitel}
  For $E \in \DF^{\txt{el}}$, the map $E \to \tau^*(\tau E)$ in
  $\mathcal{P}(\DFi)$ is an equivalence.
\end{lemma}
\begin{proof}
  First consider the case where $E = ([0], \mathbf{1})$, so that $\tau
  E = \eta$. For any map of trees $\varphi: T \to \eta$, the tree $T$
  must be \emph{linear}, i.e. have only unary vertices. But then $T =
  \tau([n], \mathbf{1} = \cdots = \mathbf{1})$ for some $n$ and
  $\varphi = \tau(\psi)$ for the unique map $\psi: ([n], \mathbf{1} =
  \cdots = \mathbf{1}) \to ([0], \mathbf{1})$ in $\DFi$. It follows
  that $\tau^*\eta = ([0],\mathbf{1})$. The argument for $E = ([1],
  \mathbf{k} \to 1)$ is similar. Note that $\tau E$ is the corolla
  $C_k$. Consider an $X \in \DFi$ with a map $\tau X \to \tau E$. If
  it is not surjective, then it factors as $\tau X \to \eta \to \tau
  E$, where $\eta \to \tau E$ is the inclusion of some edge of $C_k$,
  and one reduces to the previous case to see that $\tau X \to \tau E$
  is the image of the unique map $X \to E$ in $\DFi$. If $\tau X \to
  \tau E$ is surjective, then clearly $X$ must be of the form $([n],
  \mathbf{k} \simeq \cdots \simeq \mathbf{k} \to \mathbf{1} = \cdots =
  \mathbf{1})$ for some $n\geq 1$. Again one observes that there is a
  unique map $X \to E$ whose image is $\tau X \to \tau E$, which
  implies the lemma.
\end{proof}

\section{Proof of the Comparison Result}\label{sec:comparison}
Our goal in this section is to prove that the \icats{} $\PSegO$ and
$\PSegDFi$ are equivalent. More precisely, we saw in the previous
section that the map $\tau \colon \DFi \to \bbO$ induces a functor
between the \icats{} of Segal presheaves, and we will show that this
gives the desired equivalence:
\begin{thm}\label{thm:taueq}
  The functor $\tau^{*} \colon \PSegO \to \PSegDFi$ is an equivalence
  of \icats{}.
\end{thm}

Since $\tau$ preserves inert-active factorizations, it restricts to a
functor $\tau_{\txt{int}} \colon \DFiint \to \bbOint$, and we have a
commutative diagram
\csquare{\PSegO}{\PSegDFi}{\PSeg(\bbOint)}{\PSeg(\DFiint),}{\tau^{*}}{j_{\bbO}^*}{j_{\DFi}^{*}}{\tau_{\txt{int}}^{*}}
where $j_{\bbO}$ and $j_{\DFi}$ denote the inclusions $\bbOint \to
\bbO$ and $\DFiint \to \DFi$, and $\PSeg(\DFiint)$ and
$\PSeg(\bbOint)$ denote the full subcategories of
$\mathcal{P}(\DFiint)$ and $\mathcal{P}(\bbOint)$ spanned by the
presheaves that are right Kan extensions of their restrictions to
$\DFel$ and $\bbOel$, respectively.

\begin{lemma}\label{lem:tauint}
  The functor $\tau_{\txt{int}}^{*} \colon \PSeg(\bbOint) \to
  \PSeg(\DFiint)$ is an equivalence.
\end{lemma}
\begin{proof}
  Consider the commutative square
  \csquare{\PSeg(\bbOint)}{\PSeg(\DFiint)}{\mathcal{P}(\bbO^{\txt{el}})}{\mathcal{P}(\DF^{\txt{el}}).}{\tau_{\txt{int}}^{*}}{}{}{\tau|^{*}_{\DF^{\txt{el}}}}
  The map $\tau$ restricts to an equivalence $\DF^{\txt{el}} \to
  \bbO^{\txt{el}}$ by Lemma~\ref{lem:taueleq}, so the bottom
  horizontal map here is an equivalence. Moreover, the vertical maps
  are equivalences by \cite{HTT}*{Proposition 4.3.2.15}, since
  $\PSeg(\bbOint)$ and $\PSeg(\DFiint)$ are by definition the \icats{}
  of presheaves that are right Kan extensions of presheaves on
  $\bbO^{\txt{el}} \simeq \DF^{\txt{el}}$. By the 2-out-of-3 property,
  it follows that the top horizontal map $\tau_{\txt{int}}^{*}$ is
  also an equivalence.
\end{proof}

\begin{lemma}\ 
  \begin{enumerate}[(i)]
  \item The functor $j_{\bbO}^{*} \colon \PSegO \to \PSeg(\bbOint)$
    has a left adjoint $F_{\bbO} := L_{\bbO}j_{\bbO,!}$, and the
    adjunction $L_{\bbO}j_{\bbO,!} \dashv j_{\bbO}$ is monadic.
  \item The functor $j_{\DFi}^{*} \colon \PSegDFi \to \PSeg(\DFiint)$
    has a left adjoint $F_{\DFi} := L_{\DFi}j_{\DFi,!}$, and the
    adjunction $L_{\DFi}j_{\DFi,!} \dashv j_{\DFi}$ is monadic.
  \end{enumerate}
\end{lemma}
\begin{proof}
  We will prove (i); the proof of (ii) is the same. The existence of
  the left adjoint $L_{\bbO}j_{\bbO,!}$ is obvious, so by
  \cite{HA}*{Theorem 4.7.4.5} it remains to show that $j_{\bbO}^{*}$
  detects equivalences and that $j_{\bbO}^{*}$-split simplicial
  objects in $\PSegO$ have colimits and these are preserved by
  $j_{\bbO}^{*}$. Since $\bbOint$ contains all the objects of $\bbO$
  it is clear that $j_{\bbO}^{*}$ detects equivalences, and we also
  know that $\PSegO$ has small colimits. Suppose then that we have a
  $j_{\bbO}^{*}$-split simplicial object $X_{\bullet}$ in $\PSegO$,
  i.e. $j_{\bbO}^{*}X_{\bullet}$ extends to a split simplicial object
  $X'_{\bullet} \colon \simp_{-\infty}^{\op} \to \PSeg(\bbOint)$. If
  we consider $X_{\bullet}$ as a diagram in $\mathcal{P}(\bbO)$ with
  colimit $X$, then the colimit of $X_{\bullet}$ in $\PSegO$ is
  $L_{\bbO}X$. On the other hand, the colimit $X$ is preserved by
  $j^{*}_{\bbO} \colon \mathcal{P}(\bbO) \to \mathcal{P}(\bbOint)$
  (since this functor is a left adjoint). But by \cite{HA}*{Remark
    4.7.3.3}, the diagram $X'_{\bullet}$ is a colimit diagram also
  when viewed as a diagram in $\mathcal{P}(\bbOint)$, so
  $j^{*}_{\bbOint}X \simeq X'_{-\infty}$. This means that the presheaf
  $X$ satisfies the Segal condition, and so $X \simeq L_{\bbO}X$,
  i.e. $X$ is also the colimit of $X_{\bullet}$ in
  $\PSeg(\bbO)$. Since its image in $\PSeg(\bbOint)$ is
  $X'_{-\infty}$, this colimit is indeed preserved.
\end{proof}

The two preceding lemmas imply that $\PSegO$ and $\PSegDFi$ are both
the \icats{} of algebras for monads on $\mathcal{P}(\DF^{\txt{el}})
\simeq \mathcal{P}(\bbO^{\txt{el}})$. To show that these \icats{} are
the same, it will therefore be sufficient to prove that these two
monads are equivalent. Our proof of this makes use of the existence of
a \emph{right} adjoint to $\tau^{*}$:
\begin{propn}\label{propn:rightadj}
  The functor $\tau_{*}$ given by right Kan extension along $\tau$
  restricts to a functor \[\tau_{*} \colon \PSegDFi \to \PSegO,\] right
  adjoint to $\tau^{*}$.
\end{propn}

Let us show how to deduce Theorem~\ref{thm:taueq} from this; the
remainder of this section is then devoted to proving
Proposition~\ref{propn:rightadj}.

\begin{lemma}\label{lem:itaueq}
  The canonical map $\tau^{*}_{\txt{int}}j_{\bbO}^{*}\tau_{*} \simeq
  j_{\DFi}^{*}\tau^{*}\tau_{*} \to j_{\DFi}^{*}$ is a natural equivalence.
\end{lemma}
\begin{proof}
Recall the commutative diagram
\[
\begin{tikzcd}
\PSeg(\bbO) \arrow{d}{j^*_\bbO} \arrow{r}{\tau^*} & \PSeg(\DFi) \arrow{d}{j_{\DFi}^*} \\
\PSeg(\bbOint) \arrow{r}{\tau_{\txt{int}}^*} \arrow{d} & \PSeg(\DFiint) \arrow{d} \\
\mathcal{P}(\bbO^{\txt{el}}) \arrow{r} & \mathcal{P}(\DF^{\txt{el}}).
\end{tikzcd}
\]
We saw in the proof of Lemma \ref{lem:tauint} that the lower two
vertical arrows are equivalences. Therefore it suffices to check that
for $\mathcal{F} \in \PSegDFi$ and $E \in \DF^{\txt{el}}$ the natural
map $\tau^*\tau_*\mathcal{F}(E) \rightarrow \mathcal{F}(E)$ is an
equivalence. We may identify the domain of this map as
\[(\tau_{*}\mathcal{F})(\tau E) \simeq \lim_{X \in (\DFiop)_{\tau E/}}
\mathcal{F}(X),\] where $(\DFiop)_{\tau E/} \simeq ((\DFi)_{/\tau
  E})^{\op}$ and $(\DFi)_{/\tau E} := \DFi \times_{\bbO} \bbO_{/\tau
  E}$. But the unit morphism $E \to \tau^{*}\tau_{!}E \simeq
\tau^{*}(\tau E)$ is an equivalence for $E \in \DF^{\txt{el}}$ by
Lemma \ref{lem:unitel}, hence $(E, \tau E = \tau E)$ is a terminal
object in $(\DFi)_{/\tau E}$. Therefore it is initial in
$(\DFiop)_{\tau E/}$ and this implies the map is an equivalence.
\end{proof}

\begin{proof}[Proof of Theorem~\ref{thm:taueq}]
  By \cite{HA}*{Corollary 4.7.4.16} it suffices to show that the
  canonical natural transformation $F_{\DFi}\circ \tau^{*}_{\txt{int}} \to
  \tau^{*}F_{\bbO}$ is an equivalence. But by
  Proposition~\ref{propn:tau*segeq} these functors are both left
  adjoints, and so we have an equivalence of left adjoints \IFF{} the
  corresponding transformation of right adjoints $j_{\bbO}^{*}\tau_{*}
  \to (\tau^{*}_{\txt{int}})^{-1} j_{\DFi}^{*}$ is an
  equivalence. This now follows from Lemma~\ref{lem:itaueq}.
\end{proof}

Proposition~\ref{propn:rightadj} is an immediate consequence of the
following result, to which we now turn:
\begin{propn}\label{propn:tau*segeq}
  The functor $\tau^{*} \colon \mathcal{P}(\bbO) \to
  \mathcal{P}(\DFi)$ preserves Segal equivalences.
\end{propn}
Our proof of Proposition~\ref{propn:tau*segeq} is based on the proof
of \cite{HeutsHinichMoerdijkDendrComp}*{Proposition 5.5.9}. Before we
give it, we must introduce some notation and prove two technical
lemmas:

\begin{defn}
  For $T \in \bbO$ a tree with at least two vertices, let $\dext T$
  denote the \emph{external boundary} of $T$, namely the presheaf on
  $\bbO$ constructed as the union of all the \emph{external faces} of
  $T$. To be precise, let $\txt{Sub}(T)$ be the
  full subcategory of $(\bbOint)_{/T}$ on the proper subtrees of $T$ and define
  $\dext T$ to be the colimit of the composition $\txt{Sub}(T) \to
  \bbO \to \mathcal{P}(\bbO)$.
\end{defn}

\begin{lemma}\label{lem:dextTSeg}
  For $T$ in $\bbO$ with at least two vertices, let $(\dext T)_{\Seg}$ denote the colimit
  of the functor \[\txt{Sub}(T) \to \mathcal{P}(\bbO), \quad
  S \mapsto S_{\Seg}\] (this is well-defined
  since the maps in $\bbO$ involved are all inert). Then the natural
  map $(\dext T)_{\Seg} \to T_{\Seg}$ is an equivalence.
\end{lemma}
\begin{proof}
  Let $\mathcal{I} \to \txt{Sub}(T)$ denote the
  Grothendieck opfibration associated to the functor sending $S$ to
  $\bbOel_{/S}$. By \cite{enrbimod}*{Corollary 5.7} we
  can regard $(\dext T)_{\Seg}$ as the colimit of the functor
  $\mathcal{I} \to \mathcal{P}(\bbO)$ sending
  $(S, (E \to S) \in \bbOel_{/S})$ to $E$, and
  the map $(\dext T)_{\Seg} \to T_{\Seg}$ is the map on colimits
  induced by the functor $\Phi \colon \mathcal{I} \to \bbOelT$ that
  takes $(S, E \to S)$ to $E \to S \to T$. It
  therefore suffices to prove that $\Phi$ is cofinal. But $\Phi$ admits a left adjoint, given by the functor sending $E \rightarrow T$ to $(E, E = E)$.  
\end{proof}

For the following definition and lemma it will be clearer to work with
(Segal presheaves on) $\DF$ rather than $\DFi$; this makes no
difference due to Lemma~\ref{lem:ieq}.

\begin{defn}
  Let $\mathrm{sd}(\Delta^n)$ denote the partially ordered set of faces of
  $\Delta^{n}$ (meaning injective maps $[m] \hookrightarrow [n]$ in
  $\simp$) or equivalently the partially ordered set of non-empty
  subsets of $\{0,\ldots,n\}$; in other words, $\mathrm{sd}(\Delta^n)$ is the \emph{barycentric subdivision} of $\Delta^n$.
  We will denote the subset
  $\{i_{1},\ldots,i_{k}\}$ where $i_{1} < i_{2} < \cdots <
  i_{k}$ by $(i_{1},\ldots,i_{k})$. Given a full subcategory
  (i.e. partially ordered subset) $\mathcal{G} \subseteq \mathrm{sd}(\Delta^n)$
  and $X \in \DF$ of length $n$, let $X(\mathcal{G})$ denote the
  colimit in $\mathcal{P}(\DF)$ over $\varphi \in \mathcal{G}$ of $\varphi^{*}X$. For
  $\mathrm{sd}(\Lambda^n_i)$ the subcategory containing all objects except
  $(0,\ldots,n)$ and $(0,\ldots,i-1,i+1,\ldots,n)$, we write
  $\Lambda^{n}_{i}X$ for $X(\mathrm{sd}(\Lambda^n_i))$.
\end{defn}

\begin{lemma}\label{lem:hornSegeq}
	For any $X \in \DF$ of length $n$, the map $\Lambda^{n}_{n-1}X \to X$
	is a Segal equivalence.
\end{lemma}
\begin{proof}
	By the 2-out-of-3 property, it suffices to show that the map
	$X_{\Seg} \to \Lambda^{n}_{n-1}X$ is a Segal equivalence. To prove
	this, we consider the following filtration of $\mathrm{sd}(\Lambda^n_{n-1})$:
	we let $\mathcal{G}_{d} \subseteq \mathrm{sd}(\Lambda^n_{n-1})$ contain all
	subsets of length $\leq d$ together with those of length $d+1$ that are
	of the form $(i_{0},\ldots,i_{d-1}, i_{d}-1, i_{d})$. Then
	$\mathcal{G}_{n-2} = \mathrm{sd}(\Lambda^n_{n-1})$ and we have a filtration
	\[ X_{\Seg} \to X(\mathcal{G}_{0}) \to X(\mathcal{G}_{1}) \to
	\cdots \to X(\mathcal{G}_{n-2}) \simeq \Lambda^{n}_{n-1}X.\]
	
	It thus suffices to show that the maps $X_{\Seg} \to
	X(\mathcal{G}_{0})$ and $X(\mathcal{G}_{d-1}) \to
	X(\mathcal{G}_{d})$ ($d = 1,\ldots, n-2$) are Segal equivalences.
	
	Note that the map $X_{\Seg} \to X(\mathcal{G}_0)$ is an equivalence by construction; indeed, $\mathcal{G}_0$ consists of the subsets with one element, together with the subsets of the form $(i,i+1)$. To see that $X(\mathcal{G}_{d-1}) \to X(\mathcal{G}_{d})$ is a
	Segal equivalence, we consider a filtration
	\[ \mathcal{G}_{d-1} = \mathcal{H}_{d}^{d} \subseteq
	\mathcal{H}^{d+1}_{d} \subseteq \cdots \subseteq \mathcal{H}^{n}_{d}
	= \mathcal{G}_{d},\] where $\mathcal{H}_{d}^{j}$ is the full subcategory containing
	$\mathcal{G}_{d-1}$ together with those objects $(i_{0},\ldots,i_{k})$ of $\mathcal{G}_d$ with $i_k \leq j$.
	Let $T^{j}_{d}$ denote the objects of length $d+1$ in
	$\mathcal{H}_{d}^{j}$ that do not lie in
	$\mathcal{H}_{d}^{j-1}$. Note that for every $\sigma \in T^{j}_{d}$
	the $d$-dimensional face $d_{i}\sigma$ lies in $\mathcal{H}_{d}^{j-1}$ for
	$i \neq d$, while $d_{d}\sigma$ does \emph{not} lie in
	$\mathcal{H}_{d}^{j-1}$. Using \cite{HTT}*{Corollary 4.2.3.10} we
	therefore have pushout squares
	\nolabelcsquare{\coprod_{\sigma \in T^{j}_{d}}
		\Lambda^{d+1}_{d}\sigma^{*}X}{\coprod_{\sigma \in T^{j}_{d}}
		\sigma^{*}X}{X(\mathcal{H}_{d}^{j-1})}{X(\mathcal{H}_{d}^{j}).}
	Since  pushouts of Segal equivalences are again Segal equivalences,
	by inducting on $n$ this completes the proof.
\end{proof}

\begin{proof}[Proof of Proposition~\ref{propn:tau*segeq}]
  It suffices to show that the images under $\tau^*$ of the generating Segal
  equivalences $T_{\Seg}\to T$ for $T \in \bbO$ are Segal equivalences
  in $\mathcal{P}(\DFi)$. We will prove this by induction on the
  number of vertices in $T$, noting that if $T$ is $\eta$ or $T$ has one
  vertex, i.e. $T \in \bbO^{\txt{el}}$, then the statement is
  vacuous. Given $T \in \bbO$ with two or more vertices, we have a commutative square
  \nolabelcsquare{(\dext T)_{\Seg}}{\dext T}{T_{\Seg}}{T.}  Here the
  left vertical map is an equivalence by Lemma~\ref{lem:dextTSeg}, and
  the top horizontal map is the colimit over $S \in
  \txt{Sub}(T)$ of the maps $S_{\Seg} \to
  S$. Since $\tau^{*}$ preserves colimits and $S$ has fewer vertices than $T$ for all $S \in
  \txt{Sub}(T)$, we know by the inductive hypothesis that
  $\tau^{*}$ of this map is a Segal equivalence. By the 2-out-of-3
  property, to show that
  $\tau^{*} T_{\Seg} \to \tau^{*}T$ is a Segal equivalence it
  therefore suffices to show that $\tau^{*}(\dext T) \to \tau^{*} T$ is
  a Segal equivalence.

  To prove this, we will consider a filtration on $\tau^{*}T$. In
  order to define this we must first introduce some terminology; let
  us say that a map $\varphi \colon X \to \tau^{*}T$ is
  \emph{non-degenerate} if it does not factor through any non-trivial
  surjections in $\DFi$ --- more precisely, we require that for every
  factorization $X \xto{\psi} Y \to \tau^{*}T$ with $\psi$ a
  surjective map in $\DFi$, the map $\psi$ must be an isomorphism. We
  then say that $\varphi$ is \emph{admissible} if it is non-degenerate
  and preserves the root vertex --- more precisely, recall that if $X
  = ([n], f)$, then the adjunct map $\tau(X) \to T$ is a
  map of polynomial endofunctors
  \[
  \begin{tikzcd}
    \coprod_{i = 0}^{n} f(i) \arrow{d}& \coprod_{i = 0}^{n-1} f(i)
    \arrow{l} \arrow{r} \arrow{d}& \coprod_{i = 1}^{n} f(i)
    \arrow{r}\arrow{d} & \coprod_{i = 0}^{n} f(i) \arrow{d}\\
    T_{0} & \txt{sub}'(T) \arrow{r} \arrow{l} &  \txt{sub}(T) \arrow{r} & T_{0};
  \end{tikzcd}
  \]
  we say that $\varphi$ is admissible if it is non-degenerate and the map
  $\coprod_{i = 1}^{n} f(i) \to \txt{sub}(T)$ takes the root vertex of
  $X$, i.e. $f(n) = \mathbf{1}$, to the root corolla of $T$ viewed as a
  subtree of $T$. 

  We now define a subpresheaf $F_n$ of the (discrete) presheaf $\tau^*T$ as follows: $F_n(X)$ is the union of the image of $\tau_*(\dext T)(X)$ with the maps $X \to \tau^*T$ that factor through an admissible morphism  $Y \to \tau^{*}T$ with $Y$ of length $\leq n$.

  Every map $\tau(Y) \to T$ with $Y \in \DFi$ factors through an
  admissible map, so $\tau^{*}T \simeq
  \colim_{n \to \infty} F_{n}$; it hence suffices to show that the
  inclusions $F_{n-1} \hookrightarrow F_{n}$ are all Segal
  equivalences. Let $S_{n}$ denote the set of isomorphism classes of
  admissible maps $\varphi \colon \tau(X) \to T$ where $X \in \DFi$ is of
  length $n$. For such a $\varphi$, we have that:
  \begin{itemize}
  \item By the assumption that $\varphi$ is non-degenerate, the faces $d_{i}^{*}X \to X \to \tau^{*}T$ with $i = 0,n$
    factor through $\tau^{*}(\dext T)$, and so in particular through
    $F_{n-1}$.
  \item The faces $d_{i}^{*}X \to X \to \tau^{*}T$ with $0 < i < n-1$
    are admissible of length $n-1$ and so factor through $F_{n-1}$.
  \item The face $d_{n-1}^{*}X \to X \to \tau^{*}T$ is \emph{not}
    admissible --- if it were, then it is straightforward to see that
    $\varphi$ must be degenerate, which is not the case by assumption.
  \end{itemize}
  Note also that if $\varphi \colon X \to \tau^{*}T$ is non-degenerate
  and doesn't factor through $\dext T$, but is not admissible, with
  $X$ of length $n-1$, then there exists (up to isomorphism) a
  \emph{unique} admissible map $\varphi' \colon X' \to \tau^{*}T$ with
  $X'$ of length $n$ such that $d_{n-1}^{*}X' \to X' \to \tau^{*}T$
  equals $\varphi$. Choosing representatives for the elements of
  $S_{n}$ therefore gives a pushout diagram
  \nolabelcsquare{\coprod_{S_{n}}
    \Lambda^{n}_{n-1}X}{F_{n-1}}{\coprod_{S_{n}} X}{F_{n}.}  Here the
  left vertical morphism is a Segal equivalence by
  Lemma~\ref{lem:hornSegeq}, hence so is the right vertical morphism.
\end{proof}

\section{Completion}\label{sec:complete}

\begin{defn}
  Let $u \colon \simp \hookrightarrow \DFi$ denote the fully faithful
  inclusion given by sending $[n]$ to $([n], \mathbf{1} = \mathbf{1} =
  \cdots =\mathbf{1})$. If $\mathcal{F} \colon \DFiop \to \mathcal{S}$
  is a Segal presheaf, then $u^{*}\mathcal{F}$ is a Segal space in the
  sense of Rezk~\cite{RezkCSS}. We say that $\mathcal{F}$ is
  \emph{complete} if the Segal space $u^{*}\mathcal{F}$ is
  complete. Similarly, we say a Segal presheaf $\mathcal{F} \colon
  \DF^{\op} \to \mathcal{S}$ is \emph{complete} if
  $u^{*}i^{*}\mathcal{F}$ is a complete Segal space, and that a Segal
  presheaf $\mathcal{F} \colon \bbO^{\op}\to \mathcal{S}$ is complete
  \IFF{} $u^{*}\tau^{*}\mathcal{F}$ is a complete Segal space. We
  write $\PCSDFi$, $\PCSDF$, and $\PCSO$ for the full subcategories of
  $\PSegDFi$, $\PSegDF$, and $\PSegO$, respectively, spanned by the
  complete Segal presheaves.
\end{defn}

\begin{remark}
  The \icats{} $\PCSDFi$, $\PCSDF$, and $\PCSO$ are accessible
  localizations of $\PSegDFi$, $\PSegDF$, and $\PSegO$,
  respectively. In particular, the inclusions $\PCSDFi \hookrightarrow
  \PSegDFi$, $\PCSDF \hookrightarrow \PSegDF$, and $\PCSO
  \hookrightarrow \PSegO$ all have left adjoints.
\end{remark}

Putting together our results from the previous sections, we get:
\begin{cor}\label{cor:itauCSeq}
  Composition with the functors $i$ and $\tau$ give equivalences of
  \icats{}
  \[ \PCSO \isoto \PCSDFi \isofrom \PCSDF.\]
\end{cor}
\begin{proof}
  Immediate from Theorem~\ref{thm:taueq}, Lemma~\ref{lem:ieq}, and
  the definition of complete Segal presheaves.
\end{proof}

Using results of Cisinski and Moerdijk in the context of dendroidal
Segal spaces, this allows us to characterize the morphisms that are
local equivalences with respect to the complete objects as the fully
faithful and essentially surjective morphisms, in the following sense:
\begin{defn}
  A morphism $\varphi \colon \mathcal{F} \to \mathcal{G}$ of Segal
  presheaves on $\bbO$ is \emph{fully faithful} if for every $n$ the
  commutative square
  \nolabelcsquare{\mathcal{F}(C_{n})}{\mathcal{G}(C_{n})}{\mathcal{F}(\eta)^{\times
      (n+1)}}{\mathcal{G}(\eta)^{\times (n+1)}}
  is a pullback square.
  We say $\varphi$ is \emph{essentially surjective} if the morphism
  $u^{*}\tau^{*}\mathcal{F} \to u^{*}\tau^{*}\mathcal{G}$ of Segal
  spaces is essentially surjective. Obvious variants of this
  definition also give notions of fully faithful and essentially
  surjective morphisms between Segal presheaves on $\DF$ and $\DFi$.
\end{defn}

\begin{cor}\label{cor:complete}
  A morphism of Segal presheaves (on $\DF$, $\DFi$, or $\bbO$) maps to
  an equivalence of complete Segal presheaves \IFF{} it is fully
  faithful and essentially surjective. In other words, the
  localization functors from Segal presheaves to complete Segal
  presheaves exhibit the latter as the localization of the Segal
  presheaves at the fully faithful and essentially surjective functors.
\end{cor}
\begin{proof}
  For $\bbO$, this holds by \cite[Theorem
  8.11]{CisinskiMoerdijkDendSeg}. The other two cases then follow from
  Theorem~\ref{thm:taueq}, Lemma~\ref{lem:ieq}, and the definitions of
  complete objects and fully faithful and essentially surjective morphisms.
\end{proof}

\end{document}